\documentclass[a4paper,12pt,twoside]{article}
\usepackage{amsmath,amsfonts,amssymb}

\usepackage{amsbsy}
\usepackage{amssymb}
\usepackage{amsfonts}
\usepackage{amsmath}
\usepackage{mathrsfs}

\usepackage{amsopn}
\usepackage{amsthm}
\usepackage[english]{babel}
\usepackage[applemac]{inputenc}
\usepackage[T1]{fontenc}

\usepackage{graphicx}

\newtheorem{teo}{Theorem}
\newtheorem{lemma}{Lemma}

\newtheorem{cor}{Corollary}

\newtheorem{defi}{Definition}

\theoremstyle{definition}
\newtheorem{remark}{Remark}

\theoremstyle{definition}

\theoremstyle{definition}

\theoremstyle{remark}

\DeclareMathOperator{\re}{\mathbb{R}}

\DeclareMathOperator{\er}{\mathbf{E}}
\DeclareMathOperator{\pr}{\mathbf{P}}

\DeclareMathOperator{\p}{\mbox{\rm I\hspace{-0.02in}P}}
\DeclareMathOperator{\e}{\mbox{\rm I\hspace{-0.02in}E}}

\newcommand{\R}{\mbox{\rm I\hspace{-0.02in}R}}
\DeclareMathOperator{\alphab}{\boldsymbol\alpha}

\title{Entrance laws for positive self-similar Markov processes}
\author{V{\'\i }ctor Rivero\thanks{Centro de
Investigaci\'on en Matem\'aticas (CIMAT A.C.) E-mail:
rivero@cimat.mx}}
\date{}

\begin{document}
\maketitle
\begin{abstract}
In this paper we propose an alternative construction of the self-similar entrance laws for positive self-similar Markov processes. The study of entrance laws has been carried out in previous papers using different techniques,  depending on whether the process hits zero in a finite time almost surely or not. The technique here used allows to obtain the entrance laws in a unified way. Besides, we show that in the case where the process hits zero in a finite time, if there exists a self-similar entrance law, then there are infinitely many, but they can all be embedded into a single one. We propose a pathwise extension of this embedding for self-similar Markov processes. We apply the same technique to construct entrance law for other types self-similar processes.    
\end{abstract}


\section{Introduction and main result}
Let $\p=(\p_x, x\geq 0)$ be a family of probability measures on
Skorohod's space $\mathbb{D}^+,$ the space of c\`adl\`ag paths
defined on $[0,\infty[$ with values in $\R^+$. The space
$\mathbb{D}^+$ is endowed with the Skohorod topology and $\mathcal{D}$ is its Borel
$\sigma$-field. We will denote by $X$ the canonical process of the
coordinates and $(\mathcal{G}_t, t\geq 0)$ will be the completed natural
filtration generated by $X$. Assume that under $\p$ the canonical
process $X$ is a positive self-similar Markov process (pssMp), that
is to say that $(X,\p)$ is a $[0,\infty[$-valued strong Markov
process and that it has the scaling property: there exists an
$\alpha>0$ such that for every $c>0,$
\begin{equation*}
\left(\{cX_{tc^{-\alpha}},t\geq 0\}, \p_x\right)
\stackrel{\text{Law}}{=} \left(\{X_{t},t\geq 0\},
\p_{cx}\right)\qquad \forall x\geq 0.
\end{equation*}
In this case we will say that $X$ is an $1/\alpha$-positive self-similar Markov process ($\alpha$-pssMp). We will assume furthermore that $(X,\p)$ is a pssMp for which $0$ is a cemetery state. The hitting time of zero will be denoted by $T_0=\inf\{t>0: X_{t}=0\}$. So, the law $\p_0$ will be understood as the law of the degenerated path equal to $0.$

The importance of self-similar Markov processes resides in the fact, shown by Lamperti \cite{lamperti-62}, that it is the totality of Markov processes that can arise as scaling limits of stochastic processes. Said otherwise, this the class of possible limit Markov processes that can occur upon subjecting a fixed stochastic process to infinite contractions of its space and time scales. Further information about this class of processes and its applications can be found in the review paper~\cite{pardo+rivero}  and  chapter~13 in \cite{kyprianoubook}.  One particular feature of pssMp is that they are in bijection with real-valued L\'evy processes. This useful bijection, that we will next explain, is given through the so-called Lamperti's transformation, in honour to the celebrated work of Lamperti~\cite{Lamperti}.

A $\re\cup\{-\infty\}$-valued L\'evy process is an stochastic process whose paths are c\`adl\`ag, the state $\{-\infty\}$ is an absorbing point, and it has stationary and independent increments. The state $\{-\infty\}$ is understood as an isolated point and hence the process hits this state and dies at an independent exponential time $\zeta,$ with some parameter $q\geq 0,$ the case $q=0$ is included to allow this time to be infinite a.s. The law of $\xi$ is characterized completely by its L\'evy-Khintchine exponent $\Psi,$ which takes the following form
\begin{equation}\label{Levy-K}
\log \er\left[e^{z \xi_{1}}, 1<\zeta\right]=\Psi(z)=-q+bz +\frac{\sigma^{2}}{2}z^{2}+\int_{-\infty}^{\infty}\left(e^{zy} - 1 - zy\mathbb{I}_{\{|y|<1\}}\right)\Pi(dy), 
\end{equation}
for any $z\in i\R,$ where $\sigma, b \in \R$ and $\Pi$ is a L\'evy measure  satisfying the condition $\int_{\mathbb{R}}(y^{2}\wedge 1)\Pi(dy)<\infty$. For background about L\'evy processes see \cite{bertoinbook}, \cite{kyprianoubook}, \cite{Satobook}.

In order to state our main results we recall first a few facts about self-similar Markov processes, L\'evy processes and exponential functionals of L\'evy processes.  It is well known (see Lamperti \cite{Lamperti}) that for any $1/\alpha$-pssMp, $X=(X_{t}, t\geq 0),$ there exists a $\re\cup\{-\infty\}$ valued Lévy process $\xi$ independent of the starting point $X_0,$ such that 
\begin{equation}
\label{Lamperti}
X_t1_{\{t<T_{0}\}}=X_0\exp\left(\xi_{\tau(tX_0^{-\alpha})}\right)1_{\{\tau(tX_0^{-\alpha})<\zeta\}}, \quad t \geq 0,  
\end{equation}
where  $\tau$ is the time-change $$\qquad \tau(t)=\inf\left\{s>0: \int^s_{0}\exp(\alpha \xi_{u}) \mathrm du>t\right\},\qquad t\geq 0,$$ with the usual convention $\inf\{\emptyset\}=\infty$. 
Lamperti  proved that, for any $x>0$, $T_0$ is finite $\mathbb P_x$-a.s. if and only if either $\zeta<\infty$ a.s., or $\zeta=\infty$ a.s. and $\lim_{t \rightarrow \infty} \xi_t = - \infty$  a.s. Conversely, Lamperti showed that given a L\'evy process $\xi,$ the transformation just described gives rise to a $1/\alpha$-pssMp. We will refer to this transformation as Lamperti's transformation. Throughout this paper we will assume $\p$ is the reference measure, and under $\p$, $X$ will be a pssMp and $\xi$ the L\'evy process associated to it via Lamperti's transformation. The measures $(\p_{x},x>0)$ are the a conditional regular version of the law of $X$ given $X_{0}=x.$ Notice that under $\p_{x},$ $\xi$ starts from $\log x.$ This implies that the law of $\xi,$ under $\p_{x},$ is that of $\xi+\log x$ under $\p.$  Besides, it follows from Lamperti's transformation that under $\p_{x}$ the first hitting time of $0$ for $X,$ $T_{0}$, has the same law as $x^{\alpha}\int^{\zeta}_{0}\exp(\alpha\xi_{s})\mathrm ds,$ under $\p.$ The random variable $I,$ defined by
\begin{equation}\label{I}
I:=\int_0^{\zeta}\exp(\alpha\xi_s) \mathrm ds,
\end{equation}
is usually named \textit{exponential functional of the L\'evy process $\xi.$} Lamperti's above mentioned result implies that $I$ is a.s. finite if and only if $\zeta<\infty$ a.s. or $\zeta=\infty$ and $\lim_{t \rightarrow \infty} \xi_t = - \infty,$  a.s. 

Motivated by defining a pssMp issued from $0$, when constructed using Lamperti's transformation, there have been several papers studying the existence of what we call here  {\it self-similar entrance laws}, see for instance \cite{bertoin-caballero}, \cite{BY2002}, \cite{cc},\cite{ckpr} and \cite{pardo+rivero} where an account on this topic is provided. This is the object of main interest in this paper.  We will say that a family of sigma-finite measures on $(0,\infty)$, $\{\eta_{t}, t>0\}$, is a self-similar entrance law for the semigroup $\{P^{X}_{t}, t\geq 0\}$ of $X$ if the following are satisfied
\begin{itemize}
\item[(EL-i)] the identity between measures $$\eta_{s}P^{X}_{t}=\eta_{t+s},$$ that is $$\int_{(0,\infty)}\eta_{s}(\mathrm dx)\e_{x}\left[f(X_{t}), t<T_{0}\right]=\int_{(0,\infty)}\eta_{t+s}(\mathrm dx)f(x),$$
$\forall f:(0,\infty)\rightarrow \mathbb R$ positive measurable, 
 holds for any $s>0,$ $t\geq 0;$ 
\item[(EL-ii)]there exists an index $\gamma\geq 0$ such that for all $s>0,$ $$\eta_{s}f=s^{-\gamma/\alpha}\eta_{1}H_{s^{1/\alpha}}f,$$ where $f$ denotes any positive and measurable function, and for $c>0,$ $H_{c}$ denotes the dilation operator $H_{c}f(x)=f(cx).$
\end{itemize} In that case, we say that $\{\eta_{s}, s>0\},$ is a $\gamma$-self-similar entrance law, $\gamma$-ssel for short, associated to $X.$ Observe that the condition (EL-ii) is equivalent to the apparently more general condition: there is a $\gamma\geq 0$ such that for any $c>0,$ $s>0,$ $$\eta_{s}f=c^{-\gamma}\eta_{sc^{-\alpha}}H_{c}f$$ for any positive and measurable function $f.$

In this paper our main concern is to describe the family of $\sigma$-finite ssel for a pssMp that either hits zero in a finite time or never hits zero and the underlying Lévy process in Lamperti's transformation drifts towards $+\infty$.  

In several instances we will assume that there exists a $\theta\geq 0$ such that
$$\displaystyle \e(\exp\{\gamma\xi_{1}\}, 1<\zeta)\leq 1.$$ Which is equivalent to ask that
$$\displaystyle \e(\exp\{\gamma\xi_{t}\}, t<\zeta)\leq 1,\qquad \forall t\geq 0.$$ Under this condition we will denote by $\p^{(\theta)}$ the unique probability measure on $\mathbb{D}$ such that 
$$\p^{(\theta)}=e^{\theta\xi_{t}}\p\quad \text{on } \mathcal{F}_{t},\ \text{for all } t\geq 0.$$ For  $\theta=0,$ we will write $\p$ instead of $\p^{(0)}$. As usual, we will denote by $\widehat{\p}^{(\theta)}$ the law of the dual L\'evy process $\widehat{\xi}=-\xi$ under $\p^{(\theta)}.$ 

We have the following theorem whose proof was partially inspired by Fitzsimmons' \cite{fitzsimmons2006} constructions of excursions measures for pssMp. 
\begin{teo}\label{thm:1}
Let $\left((X_{t})_{t\geq 0},(\p_{x})_{x>0}\right)$ be a $1/\alpha$-positive self-similar Markov process  and $\xi$ the L\'evy process associated to it via Lamperti's transformation. Assume that $X$ hits zero in a finite time a.s. For $\gamma>0$ fixed, the following are equivalent
\begin{itemize}
\item[(i)] $\displaystyle \e(\exp\{\gamma\xi_{1}\},1<\zeta)\leq 1;$ 
\item[(ii)] the family of measures $(\mu^{\gamma}_{s}, s>0),$ defined by $$\mu^{\gamma}_{s}f:=s^{-\gamma/\alpha}\widehat{\e}^{(\gamma)}\left(f\left(\left(\frac{s}{I}\right)^{1/\alpha}\right)I^{\frac{\gamma}{\alpha}-1}\right),$$ for $f:\re^{+}\to\re^{+}$ measurable, forms a $\gamma$-ssel for $\left((X_{t})_{t\geq 0},(\p_{x})_{x>0}\right),$ and $\mu^{\theta}_{1}1<\infty;$  
\item[(iii)] there exists a $\gamma$-ssel $(\eta^{\gamma}_{t}, t>0)$ for $\left((X_{t})_{t\geq 0},(\p_{x})_{x>0}\right),$ such that $\eta^{\gamma}_{1}$ is a probability measure.
\end{itemize}In this case, the measures in (ii) form the unique, up to multiplicative constants, finite $\gamma$-ssel for $X.$

Furthermore,  when one of the above conditions is satisfied, it is then satisfied for every $0<\beta<\gamma.$ For any pair $(\beta^{\prime}, \beta),$ such that $0<\beta^{\prime}<\beta\leq \gamma,$ there exists a constant $0<C_{\beta^{\prime},\beta}<\infty$ such that the associated ssel are related by means of the identity
\begin{equation}\label{entrancelaws}
\mu^{\beta^{\prime}}_{s}(dx)=C_{\beta^{\prime},\beta}s^{(\beta-\beta^{\prime})/\alpha}\int_{z\in(0,1)}\p\left(\mathcal{B}_{\frac{\beta^{\prime}}{\alpha},\frac{(\beta-\beta^{\prime})}{\alpha}}\in dz\right)\mu^{\beta}_{s}(z^{1/\alpha}dx), \qquad x>0,\qquad s>0,
\end{equation}
with
$$\p\left(\mathcal{B}_{\frac{\beta^{\prime}}{\alpha},\frac{(\beta-\beta^{\prime})}{\alpha}}\in dy\right)=\frac{\Gamma\left(\frac{\beta}{\alpha}\right)}{\Gamma\left(\frac{\beta-\beta^{\prime}}{\alpha}\right)\Gamma\left(\frac{\beta^{\prime}}{\alpha}\right)}y^{\frac{\beta^{\prime}}{\alpha}-1}(1-y)^{\frac{(\beta-\beta^{\prime})}{\alpha}-1}1_{\{0<y<1\}}dy.$$
\end{teo}
In the case where $\displaystyle \e(\exp\{\gamma\xi_{1}\},1<\zeta)=1,$ the identity in (ii) has been obtained in \cite{rivero-cramer,nasc-cramer}. The second assertion in the latter Theorem implies that although there are infinitely many ssel, all can be embedded into a single one, namely that with largest self-similarity index. In Section~\ref{embed}, we will see that this can be extended to the level of stochastic processes. Besides, the second part of the Theorem has been observed  in the work \cite{haas+rivero2} using completely different techniques.  An easy extension of the results in the paper \cite{vuolle-apiala} shows that any ssel $\eta=(\eta_{t}, t>0)$ is such that either 
$\lim_{t\to 0}\eta_{t}1_{\{(a,\infty)\}}=0,$ for all $a>0,$ or $\lim_{t\to 0}\eta_{t}1_{\{(a,\infty)\}}>0,$ for all $a>0.$ And the latter holds if and only if there is a $\gamma>0$ such that 
\begin{equation}\label{jumpin}
\eta_{t}(dy)=\int_{(0,\infty)}\frac{dx}{x^{1+\gamma/\alpha}}\p_{x}(X_{t}\in dy, t<T_{0}),\qquad y>0,\ t>0;
\end{equation}
 in which case $\eta$ is a $\gamma$-ssel. The above facts and the Proposition~1 in \cite{rivero-cramer} imply the following Corollary, where a more tractable expression for the above ssel is provided. 
\begin{cor} 
Let $\left((X_{t})_{t\geq 0},(\p_{x})_{x>0}\right)$ be a $1/\alpha$-positive self-similar Markov process  and $\xi$ the L\'evy process associated to it via Lamperti's transformation. Assume that $X$ hits zero in a finite time a.s. For $\gamma>0,$ such that $\displaystyle \e(\exp\{\gamma\xi_{1}\},1<\zeta)<1,$ we have that there is a constant $c_{\gamma}\in(0,\infty)$ such that
$$c_{\gamma}\int_{(0,\infty)}\frac{dx}{x^{1+\gamma/\alpha}}\e_{x}(f(X_{s}), s<T_{0})=s^{-\gamma/\alpha}\widehat{\e}^{(\gamma)}\left(f\left(\left(\frac{s}{I}\right)^{1/\alpha}\right)I^{\frac{\gamma}{\alpha}-1}\right),$$ for any $f:(0,\infty)\to[0,\infty)$ measurable. 
\end{cor}

In the following Theorem we deal with the case of $0$-ssel.

 \begin{teo}\label{thm:20}
Let $\left((X_{t})_{t\geq 0},(\p_{x})_{x>0}\right)$ be a $1/\alpha$-positive self-similar Markov process  and $\xi$ the L\'evy process associated to it via Lamperti's transformation. Assume that $X$ never hits zero and $\xi$ drifts towards $\infty.$ The family of measures $(\mu_{t}, t> 0)$ defined by the relation $$\mu_{s}f:=\widehat{\e}\left(f\left(\left(\frac{s}{I}\right)^{1/\alpha}\right)I^{-1}\right),$$ for $f:\re^{+}\to\re^{+}$ measurable, forms a $0$-ssel for $\left(X_{t}, {t\geq 0}\right).$ 
 \end{teo}
In the case where the process $\xi$ has finite mean $0<m:=\er(\xi_{1})<\infty$ and the process is not arithmetic the form of the entrance law described in the latter theorem has been obtained in \cite{bertoin-caballero} and \cite{BY2002}. In those  papers the authors proved that the measures $\frac{1}{m}\mu_{t}$ are the weak limit of the law of $X_{t}$ under $\p_{x}$ as $x\to 0,$ so that the measures $\frac{1}{m}\mu_{t}$ are probability measures. In the case where $\xi$ has infinite mean, it can be proved that $\er(I^{-1})=\infty$ which implies that the measures $\mu_{t}$ are only $\sigma$-finite.  

The rest of the paper is organised as follows.  In Section~\ref{sect:2} we prove the Theorems \ref{thm:1} and \ref{thm:20}. The proof is given in a unified way, using results from the theory of Kusnetzov processes. In Section~\ref{embed}, we develop an embedding of stochastic processes, analogous to the identity in equation (\ref{entrancelaws}). In section \ref{extensions} we state without proof two extensions of the Theorem~\ref{thm:20}. Namely, in Subsection \ref{fluctuation} we provide an entrance law for processes similar to those appearing in \cite{ckpr}, closely related to the supremum process of pssMp. In Subsection~\ref{multi} an entrance law for the multi-self-similar Markov processes introduced in \cite{JACY}. We did not include the proof because the argument is very close to that provided for the Theorems \ref{thm:1} and \ref{thm:20}. A similar technique has also been used in the papers \cite{DDK} and \cite{pardo+panti+rivero} to construct entrance laws for real valued self-similar Markov processes.

\section{Proof of Theorems~\ref{thm:1} and \ref{thm:20}}\label{sect:2} 
The proof that (i) implies (ii) in Theorem~\ref{thm:1} and the claim in Theorem~\ref{thm:20} will be given in a unified way. We will assume  either of the following conditions on the underlying L\'evy process $\xi$
\begin{itemize}  
\item[(TH1)] there exists a $\theta>0$ such that $$\er(e^{\theta\xi_{1}}1_{\{1<\zeta\}})\leq 1.$$ 
\item[(TH2)] $\xi$ has an infinite lifetime and drifts towards $\infty.$
\end{itemize}
It is a standard fact that (TH1) is equivalent to require that $$\e(e^{\theta\xi_{t}}1_{\{t<\zeta\}})\leq 1,\qquad \forall t\geq 0,$$ see e.g. \cite{Satobook} Theorem 25.17. We will say that {\bf Cram\'er's condition} is satisfied with index $\theta>0$ if the equality holds 
$$\e(e^{\theta\xi_{t}}1_{\{t<\zeta\}})=1,\qquad \forall t\geq 0.$$ When (TH1) holds, we will denote by $\p^{(\theta)}$ the unique probability measure such that 
$$\p^{(\theta)}=e^{\theta\xi_{t}}\p\quad \text{on } \mathcal{F}_{t},\ \text{for all } t\geq 0.$$ It is easily verified that under $\p^{(\theta)}$ the canonical process still is a $\re\cup\{\infty\}$-valued L\'evy process, see e.g. \cite{Satobook} Chapter 33. Furthermore, under $\p^{(\theta)}$ the lifetime is infinite a.s. if and only if Cram\'er's condition is satisfied. Indeed, we have the equality
$$\p^{(\theta)}(t<\zeta)=\e(e^{\theta\xi_{t}}1_{t<\zeta})\leq 1,\qquad \text{for all } t\geq 0,$$ and thus if Cram\'er's condition is satisfied then, $\zeta=\infty,$ $\p^{(\theta)}$--a.s. We also have that if Cram\'er's condition is satisfied then, under $\p^{(\theta)},$ $\xi$ drifts towards $\infty,$ which in turn follows from the convexity of the mapping $\lambda\mapsto \log \e(e^{\lambda\xi_{1}}1_{\{1<\zeta\}}),$ on the set $\mathcal{C}=\{\beta\in \re: \e(e^{\beta\xi_{1}}1_{\{1<\zeta\}})<\infty\},$ see e.g. \cite{Satobook} Chapter 25. 

{\bf Now, we observe that when the condition (TH2) holds, the process $(\xi,\p)$ bears the same properties as $(\xi,\p^{(\theta)})$ does when Cram\'er's condition is satisfied with an index $\theta>0.$ Also, in this setting, Cram\'er's condition is trivially satisfied taking $\theta=0.$ This simple remark is the unifying point of the proofs of Theorems~\ref{thm:1} and \ref{thm:20}.  In order to give a unified argument, we will say that $\theta=0,$ whenever the conditions in (TH2) hold. In that case, Cram\'er's condition will be necessarily satisfied, and the respective measure $\p^{(0)}$ will be $\p$ itself, so no difference will be made. Naturally, the case $\theta>0$ will be exclusive to the setting (TH1). }

We will denote by $\widehat{\p},$ and $\widehat{\p}^{(\theta)},$ respectively, the law of the dual process $(-\xi_{t}, t\geq 0)$ under  ${\p},$ and ${\p}^{(\theta)}$ respectively.   Observe that the processes $(\xi,\p)$ and $(\xi,\widehat{\p}^{(\theta)})$ are in weak duality with respect to the measure 
$$\Lambda^{\theta}(dx):=e^{-\theta x}dx,\qquad x\in \re.$$
Indeed, let $f,g:\re\to\re^{+}$ measurable functions. Using Fubin's theorem and a change of variables we get:
\begin{equation*}
\begin{split}
\int_{\re}dx e^{-\theta x}f(x)\e_{x}(g(\xi_{t})1_{\{t<\zeta\}})&=\e\left(\int_{\re}dx e^{-\theta x}f(x)g(x+\xi_{t})1_{\{t<\zeta\}}\right)\\
&= \e\left(\int_{\re}dy e^{-\theta y}f(y-\xi_{t})e^{\theta\xi_{t}}g(y)1_{\{t<\zeta\}}\right)\\
&= \int_{\re}dy e^{-\theta y}g(y)\e\left(f(y-\xi_{t})e^{\theta\xi_{t}}1_{\{t<\zeta\}}\right)\\
&= \int_{\re}dy e^{-\theta y}g(y)\e^{(\theta)}\left(f(y-\xi_{t})1_{\{t<\zeta\}}\right)\\
&=\int_{\re}dy e^{-\theta y}g(y)\widehat{\e}^{(\theta)}_{y}\left(f(\xi_{t})1_{\{t<\zeta\}}\right).
\end{split}
\end{equation*}
It follows that the measure $\Lambda^{\theta}$ is excessive for both $(\xi,\p)$ and $(\xi,\widehat{\p}^{(\theta)})$. Moreover, by taking $f\equiv 1$ in the above identity, it is easily seen that Cram\'er's condition is satisfied for $(\xi,\p),$ with an index $\theta\geq 0,$  if and only if the measure is actually invariant for $(\xi,\p).$ Whilst the measure $\Lambda^{\theta}$ is invariant for $(\xi,\widehat{\p}^{(\theta)})$ if and only if the lifetime of $(\xi,\p)$ is infinite, as can also be seen from the above identity by taking $g\equiv 1.$ 

We denote by $(Y,\mathbb{Q}^{\theta})$ the Kusnetzov process associated to $\xi$ and $\Lambda^{\theta}$, see for instance \cite{DMfin} Chapter XIX or \cite{sharpebook}. $\mathbb{Q}^{\theta}$ is the unique sigma finite measure on $\mathbb{D}(\re,\re),$ such that 
$$\mathbb{Q}^{\theta}\left(Y_{t_{1}}\in{d}x_{1},Y_{t_{2}}\in{d}x_{2},\ldots, Y_{t_{n}}\in{d}x_{n}\right)=\Lambda^{\theta}({d}x_{1})Q_{t_{2}-t_{1}}(x_{1},{d}x_{2})\cdots Q_{t_{n}-t_{n-1}}(x_{n-1},{d}x_{n}),$$ for all  $-\infty<t_{1}<t_{2}<\cdots<t_{n}<\infty,$ and $x_{1},x_{2},\ldots,x_{n}\in\re,$ where $(Q_{t}, t\geq 0)$ denotes the transition semigroup of $(\xi,\p).$ An important fact about $(Y,\mathbb{Q}^{\theta})$ is that for any $x\in \re,$ its image measure under the translations of the path by $x,$ is the measure $e^{\theta x}\mathbb{Q}^{\theta}$. In particular, when $\theta=0$ the measure $,\mathbb{Q}^{\theta}$ is invariant under translations. Indeed, let $f_{1},\ldots,f_{n}$ be positive and measurable functions, $x\in\re$ and $-\infty<t_{1}<t_{2}<\cdots<t_{n}<\infty.$ Let $\phi_{x}f(y)=f(x+y),$ $y\in\re.$ We have the following identities that prove our claim. 
\begin{equation*}
\begin{split}
&\mathbb{Q}^{\theta}\left(f_{1}(Y_{t_{1}}+x)f_{2}(Y_{t_{2}}+x)\cdots f_{n}(Y_{t_{n}}+x)\right)\\
&=\int_{\re}{d}y_{1}e^{-\theta y_{1}}f_{1}(y_{1}+x)\er\left(f_{2}(\xi_{t_{2}-t_{1}}+y_{1}+x)\cdots f(\xi_{t_{n}-t_{1}}+y_{1}+x)\right)\\
&=e^{\theta x}\int_{\re}
{d}ze^{-\theta z}f_{1}(z)\er\left(f_{2}(\xi_{t_{2}-t_{1}}+z)\cdots f(\xi_{t_{n}-t_{1}}+z)\right)\\
&=e^{\theta x}\mathbb{Q}^{\theta}\left(f_{1}(Y_{t_{1}})f_{2}(Y_{t_{2}})\cdots f_{n}(Y_{t_{n}})\right)\\
\end{split}
\end{equation*}
Moreover, the image under time reversal of $(Y,\mathbb{Q}^{\theta}),$ at any finite time, gives a process with the same semigroup as $(\xi,\widehat{\p}^{(\theta)}),$ see \cite{DMfin} Ch. XIX- no.14.  

We denote by $\widetilde{\alpha}$ and $\widetilde{\beta}$ the birth and death times of $(Y,\mathbb{Q}^{\theta}).$ Observe that if Cram\'er's condition is satisfied with an index $\theta\geq 0,$ then $\widetilde{\alpha}=-\infty,$ $\mathbb{Q}^{\theta}$-a.s. While if $(\xi, \p)$ has an infinite lifetime then $\widetilde{\beta}=\infty,$ $\mathbb{Q}^{\theta}$-a.s. These are well known facts that come from the above observation that in these cases $\Lambda^{\theta}$ is invariant for $(\xi,\p),$ or $(\xi,\widehat{\p}^{(\theta)}),$ respectively, see e.g. Theorem 6.7 in \cite{getoorbook}. We define $(\rho(t), t\geq 0),$ by $$\rho(t)=\int^t_{\widetilde{\alpha}}e^{\alpha Y_{s}}{d}s, \qquad t\in\re.$$ Under $\mathbb{Q}^{\theta},$ the process $\rho$ is almost surely finite. This is obtained by conditioning on the future and applying Proposition 4.7 in \cite{mitro1979}, to get
$$\mathbb{Q}^{\theta}(\rho(t)=\infty,\widetilde{\alpha}<t<\widetilde{\beta})=\int_{\re}\Lambda^{\theta}({d}x)\widehat{\p}^{(\theta)}_{x}\left(\int^{\zeta}_{0}e^{\alpha\xi_{s}}{d}s=\infty\right)=0;$$ where the third identity is a consequence of the fact that the process $(\xi,\widehat{\p}^{(\theta)})$ either drifts towards $-\infty$ or has a finite lifetime, and thus $I=\int^{\zeta}_{0}e^{\alpha{\xi}_{s}}{d}s<\infty,$ $\widehat{\p}^{(\theta)}_{x}$-a.s. for $x\in\re.$ Whenever $\theta=0$, we have by the simple Markov property under $\mathbb{Q}^{\theta}$ that $\rho(\infty)=\infty,$ a.s. Indeed, given that $\rho(0)<\infty$ its suffices to prove that $\int^\infty_{0}e^{\alpha Y_{s}}{d}s$ is a.s. infinite under $\mathbb{Q}^{\theta}.$ For, we observe that when $\theta=0,$
\begin{equation*}
\begin{split}
\mathbb{Q}^{\theta}\left(\int^\infty_{0}e^{\alpha Y_{s}}{d}s<\infty, \widetilde{\alpha}<0<\widetilde{\beta}\right)=\int_{\re}\Lambda^{\theta}({d}x)\p_{x}\left(\int^\infty_{0}e^{\alpha Y_{s}}{d}s<\infty\right)=0;
\end{split}
\end{equation*}where the last identity follows from the fact that under $\p_{x}$ the L\'evy process $\xi$ drifts to $\infty.$ Let $C_{t}$ be the time change induced by $\rho,$ that is $$C_{t}=\inf\{s>0: \rho(s)>t\},\qquad t\geq0.$$ It follows from the previous discussion that $C_{t}<\infty$ $\mathbb{Q}^{\theta}$-a.s. By the theory of time changes developed by Kaspi \cite{Kaspi} it follows that the family of measures 
$$\eta^{\theta}_{t}f:=\mathbb{Q}^{\theta}\left(f\left(e^{Y_{C_{t}}}\right), 0<C_{t}<1\right), \qquad t>0,$$ is an entrance law for the pssMp $X.$ It is important to mention that for each $t>0,$ $\eta^{\theta}_{t}$ has the following scaling property. If $H_{c}$ denotes the dilation operator $H_{c}f(x)=f(cx), x\in\re,$ we have the equality for any $f:\re\to\re$ measurable and positive
\begin{equation}\label{scaling}
\eta^{\theta}_{t}H_{e^{x}}f=e^{{\theta}x}\eta_{te^{x\alpha}}f.
\end{equation}
This fact is an easy consequence of the effect of translations under $\mathbb{Q}^{\theta},$ that we mentioned above, as the following calculations show
\begin{equation}
\begin{split}
\eta^{\theta}_{t}H_{e^{x}}f&=\mathbb{Q}^{\theta}\left(f\left(e^{(x+Y)_{C_{t}}}\right), 0<C_{t}<1\right)\\
&=\mathbb{Q}^{\theta}\left(f\left(\exp\left\{(x+Y)_{C_{te^{x\alpha}}(x+Y)}\right\}\right), 0<C_{te^{x\alpha}}(x+Y)<1\right)\\
&=e^{{\theta}x}\mathbb{Q}^{\theta}\left(f\left(\exp\left\{Y_{C_{te^{x\alpha}}(Y)}\right\}\right), 0<C_{te^{x\alpha}}(Y)<1\right)\\
&=e^{{\theta}x}\eta^{\theta}_{te^{x\alpha}}f,
\end{split}
\end{equation}
where we used that $$C_{te^{x\alpha}}(x+Y):=\inf\{s>0: \int^{s}_{\widetilde{\alpha}}\exp\{\alpha(x+Y_{u})\}du>te^{x\alpha}\}=C_{t}=C_{t}(Y),\qquad x\in \re.$$ The latter fact has as a particular consequence that 
\begin{equation}\label{scaling2}
\eta^{\theta}_{t}f=t^{-\theta/\alpha}\eta^{\theta}_{1}H_{t^{1/\alpha}}f, \qquad t>0.\end{equation}  We will now prove that for $t>0$ the above constructed measures $(\eta^{\theta}_{t}, t>0)$ and the measures $(\mu^{\theta}_{t}, t>0)$  as defined in Theorem \ref{thm:1}-(ii) and Theorem~\ref{thm:20} are equal.  Let $q>0.$ Applying Fubini's theorem and inverting the time change $C,$ we obtain that for every function $f$ positive and measurable 
\begin{equation*}
\begin{split}
\int^\infty_{0}{\rm d}te^{-qt}\eta^{\theta}_{t}f&=\mathbb{Q}^{\theta}\left(\int^\infty_{0}{d}t e^{-q t}f(e^{Y_{C_{t}}})1_{\{0< C_{t}<1\}}\right)\\
&=\mathbb{Q}^{\theta}\left(\int^\infty_{-\infty}{d}s e^{\alpha Y_{s}}e^{-q \rho(s)}f(e^{Y_{s}})1_{\{0< s <1\}}\right)\\
&=\int^\infty_{-\infty}{d}s1_{\{0< s <1\}} \mathbb{Q}^{\theta}\left(e^{\alpha Y_{s}}e^{-q \rho(s)}f(e^{Y_{s}})\right).
\end{split}
\end{equation*}
Now we condition with respect to the future of $Y$, use Mitro's formula for time reversal  (Proposition 4.7 in \cite{mitro1979}), and that the marginal law of $Y$ under $\mathbb{Q}^{\theta}$ is $\Lambda^{\theta},$  to obtain the following identities
\begin{equation*}
\begin{split}
\int^\infty_{0}{\rm d}te^{-qt}\eta^{\theta}_{t}f&=\int^1_{0}{d}s\mathbb{Q}^{\theta}\left(\mathbb{Q}^{\theta}\left(e^{\alpha Y_{s}}e^{-q \rho(s)}f(e^{Y_{s}}) | \sigma\left(Y_{u}, u\geq s\right)\right)\right)\\
&=\int^1_{0}{d}s \mathbb{Q}^{\theta}\left(f(e^{Y_{s}})e^{\alpha Y_{s}}\widehat{\e}^{(\theta)}_{Y_{s}}\left(\exp\left\{{-q \int^\infty_{0}e^{\alpha {\xi}_{u}}}{d}u\right\}\right)\right)\\ 
&=\int^1_{0}{d}s\int_{\re}{d}x f(e^{x})e^{(\alpha-\theta)x}\widehat{\e}^{(\theta)}\left(\exp\{-qe^{\alpha x}I\}\right)\\
&=\int_{\re}{d}x f(e^{x})e^{(\alpha-\theta) x}\widehat{\e}^{(\theta)}\left(\exp\{-qe^{\alpha x}I\}\right).
\end{split}
\end{equation*}
Finally, an elementary change of variables, $t=e^{\alpha x}I,$ leads to
\begin{equation}\label{qpotEL} 
\int^\infty_{0}{\rm d}te^{-qt}\eta^{\theta}_{t}f=\int^{\infty}_{0}{d}te^{-q t}t^{-\theta/\alpha}\widehat{\e}^{(\theta)}\left(f\left(\left(\frac{t}{I}\right)^{1/\alpha}\right)I^{\frac{\theta}{\alpha}-1}\right).
\end{equation}
Besides, the equality in (\ref{scaling}) gives 
$$\int^\infty_{0}{\rm d}te^{-qt}\eta^{\theta}_{t}f=\int^{\infty}_{0}dte^{-qt}t^{-\theta/\alpha}\eta^{\theta}_{1}H_{t^{1/\alpha}}f.$$ Putting together the latter and former identities we get the equality of measures 
$$\frac{dt}{t^{\theta/\alpha}}\eta^{\theta}_{1}(t^{1/\alpha}dz)=\frac{dt}{t}{\mu}^{\theta}_{1}(t^{1/\alpha}dz),$$ with $\mu^{(\theta)}_{1}$ given by
$$\mu^{\theta}_{1}(dz)=\widehat{\p}^{(\theta)}\left(\left(\frac{1}{I}\right)^{1/\alpha}\in dz\right)z^{\alpha-\theta}, \qquad z>0.$$ We deduce therefrom the equality of measures 
$$\eta^{\theta}_{1}(dz)=\mu^{\theta}_{1}(dz).$$ 
The claim follows from (\ref{scaling2}). We should now justify that when $\theta>0,$ $$\mu^{\theta}_{1}1=\widehat{\e}^{\theta}(I^{\frac{\theta}{\alpha}-1})<\infty,$$ but this is a consequence of the Lemma 2 in \cite{nasc-cramer} and Lemma 3 in \cite{r-tail}, because these results ensure that this condition is implied by the condition
$$\widehat{\e}^{\theta}\left(e^{\theta\xi_{1}}1_{\{1<\zeta\}}\right)={\e}\left(1_{\{1<\zeta\}}\right)\leq 1.$$
 This finishes the proof of the implication (i) $\Rightarrow$ (ii) in Theorem \ref{thm:1} and the claim in Theorem~\ref{thm:20}. 

\subsection{Continuation of the proof of Theorem~\ref{thm:1}}
That (ii) implies (iii) is straightforward. We are left to prove that (iii) implies (i).

In the paper \cite{haas+rivero}, Lemma 5.2, it has been proved that there exists a bijection between the family of $\gamma$--ssel for a pssMp $X$, with $\gamma>0$, for which the measure corresponding to the time index $1$ is a probability measure, and the family of quasi-stationary laws for the Orstein-Uhlenbeck type process $$U:=(U_{t}=e^{-\alpha t}X_{e^{t}-1},\ 0\leq T^{U}_{0}:=\log(1+T_{0})).$$  Recall that a probability measure $\nu$ is a quasi-stationary law for $U$ if we have the equality of measures  
$$\frac{\int_{(0,\infty)}\nu(dx)\p_{x}(U_{t}\in dy, t<T_{0})}{\int_{(0,\infty)}\nu(dx)\p_{x}(t<T_{0})}=\nu(dy).$$ In that case, there exists a $\gamma>0$ such that $\int_{(0,\infty)}\nu(dx)\p_{x}(t<T_{0})=e^{-\gamma t},$ $t>0$ and 
$$\int_{(0,\infty)}\nu(dx)\p_{x}(X_{t}\in dy, t<T_{0})=e^{-\gamma t}\nu(dy).$$ The mentioned bijection is as follows. Given $\nu$ a quasi-stationary law for $U,$ as above, the family of measures defined by  
$$\eta^{\gamma}_{s}f:=s^{-\gamma/\alpha}\nu H_{s^{1/\alpha}}f, s>0.$$
constitutes a $\gamma$-ssel for $X$ such that $\eta_{1}1=1.$ Reciprocally, given a $\gamma$-ssel for $X$, such that $\eta_{1}1=1,$ the measure $\nu:=\eta_{1}$ defines a quasi-stationary law $U$. It follows that if (iii) is satisfied then there is also a quasi-stationary law $U$, which by Corollary 5.3 in \cite{haas+rivero} implies that the condition (i) in Theorem \ref{thm:1} holds. 

We will next justify the second part of Theorem~\ref{thm:1}. We assume that (i) holds for some $\beta>0$. As we mentioned before, the set $\mathcal{C}=\{\beta\in \re: \e(e^{\beta\xi_{1}}1_{\{1<\zeta\}})<\infty\},$ is convex, it contains the element $0,$ and hence the whole interval $[0,\beta].$ It follows that (i) holds for any $0<\beta^{\prime}<\beta.$ We denote by $\eta^{\beta^{\prime}}$ and $\eta^{\beta}$ the associated entrance laws for $X.$ By (iii) we know they can be taken to be such that $$\eta^{\beta^{\prime}}_{1}1=1=\eta^{\beta}_{1}1.$$  In order to get the claimed identities we make a short digression to recall further results obtained in \cite{haas+rivero}. There, it has been proved that whenever there is a $\theta>0$ such that $$\er(e^{\theta(\alpha\xi_{1})},1<\zeta)\leq 1,$$ then there is a unique in law random variable $R_{\theta}$ such that, if it is taken independent of $I,$ then 
$$IR_{\theta}\stackrel{\text{Law}}{=}\mathcal{Z}_{\theta},$$ where $\mathcal{Z}_{\theta}$ follows a Pareto distribution with parameter $\theta,$ viz.
$$\p(\mathcal{Z_{\theta}}\in dy)=\frac{\theta}{(1+y)^{1+\theta}}dy, \qquad y>0.$$ Besides, is worth noticing that elementary properties of the Beta and Gamma distributions imply that if $0<\theta^{\prime}<\theta$ and $\mathcal{B}_{\theta^{\prime},\theta-\theta^{\prime}}$ is an independent $(\theta^{\prime},\theta-\theta^{\prime})$--Beta random variable, viz. $$\p(\mathcal{B}_{\theta^{\prime},\theta-\theta^{\prime}}\in dy)=\frac{\Gamma(\theta)}{\Gamma(\theta-\theta^{\prime})\Gamma(\theta^{\prime})}y^{\theta^{\prime}-1}(1-y)^{(\theta-\theta^{\prime})-1}1_{\{0<y<1\}}dy,$$
then we have the equality in law $$\mathcal{Z}_{\theta^{\prime}}\stackrel{\text{Law}}{=}\frac{\mathcal{Z}_{\theta}}{\mathcal{B}_{\theta^{\prime},\theta-\theta^{\prime}}}.$$ We deduce therefrom the identity in law
$$R_{\theta^{\prime}}\stackrel{Law}{=}\frac{R_{\theta}}{\mathcal{B}_{\theta^{\prime},\theta-\theta^{\prime}}}.$$ This being said, we apply these facts to $\theta=\beta^{\prime}/\alpha$ and $\theta=\beta/\alpha,$ with $0<\beta^{\prime}<\beta.$ We should next relate the factors $R_{\beta^{\prime}/\alpha}$ and $R_{\beta/\alpha}$ with $\eta^{\beta^{\prime}}$ and $\eta^{\beta},$ respectively. Let $J_{\beta^{\prime}}$ and $J_{\beta}$ be random variables with law $\eta^{\beta^{\prime}}_{1}$ and $\eta^{\beta}_{1},$ respectively. Arguing as in page 482 in \cite{haas+rivero}, raplacing $\nu$ there by $\eta^{\beta^{\prime}}_{1}$ and $\eta^{\beta}_{1}$, it is proved that if these random variables are taken independent of $I$ under $\p,$ then 
$$IJ^{\alpha}_{\beta^{\prime}}\stackrel{\text{Law}}{=}\mathcal{Z}_{\beta^{\prime}}\qquad IJ^{\alpha}_{\beta}\stackrel{\text{Law}}{=}\mathcal{Z}_{\beta}.$$ Which implies $R_{\beta^{\prime}}\stackrel{\text{Law}}{=}J^{\alpha}_{\beta^{\prime}},$ $R_{\beta}\stackrel{\text{Law}}{=}J^{\alpha}_{\beta},$ and 
$$J_{\beta^{\prime}}\stackrel{\text{Law}}{=}\frac{J_{\beta}}{\mathcal{B}^{1/\alpha}_{\frac{\beta^{\prime}}{\alpha},\frac{\beta-\beta^{\prime}}{\alpha}}}.$$ With this information and the identity (\ref{scaling2}) it is easily verified that 
$$\eta^{\beta^{\prime}}_{s}f=s^{\frac{\beta-\beta^{\prime}}{\alpha}}\int_{(0,1)}\p\left(\mathcal{B}_{\frac{\beta^{\prime}}{\alpha},\frac{\beta-\beta^{\prime}}{\alpha}}\in dz\right)\eta^{\beta}_{s}H_{z^{-1/\alpha}}f,$$ for any $f$ positive and measurable. This is what the identity (\ref{entrancelaws}) states. The constant $C_{\beta^{\prime},\beta}$ in that equation appears when we remove the condition that the entrance law is such that the measure with time index $1$ is a probability measure.  Observe that the above argument also proves that for each $\gamma>0$ there is at most one $\gamma$-ssel, up to a multiplicative constant, which is given by the formula in (ii) in Theorem~\ref{thm:1}

\section{A pathwise extension of the identity (\ref{entrancelaws})}\label{embed}
Our aim in this section is to provide a pathwise explanation of the curious identity (\ref{entrancelaws}). For that end we will carryout the following program.  We will construct a self-similar process, $X^{(\beta)},$ associated to the entrance law $\eta^{\beta}$, constructed in (ii) in Theorem~\ref{thm:1}.  Then, for any pair $0<\beta^{\prime}<\beta$ we will relate, via a time change, the corresponding processes $X^{(\beta^{\prime})}$ and $X^{(\beta)}.$ We will see that we can embed the paths of $X^{(\beta^{\prime})}$ into those of $X^{(\beta)}.$

For this end, it will be necessary to assume throughout this section  that (i) in Theorem \ref{thm:1} is satisfied with some index $0<\beta<\alpha.$ Under these assumptions, the papers \cite{nasc-cramer} and \cite{fitzsimmons2006} ensure the existence of a recurrent extension $X^{(\beta)}$ of $X.$ That is, a process for which the state $0$ is a recurrent and regular state, and such that $X^{(\beta)}$ killed at its first hitting time of $0$ has the same law as $X.$  We denote by $N^{(\beta)}$ the It\^o's excursion measure from $0$ for $X^{(\beta)}.$ The measure  $N^{(\beta)}$ satisfies
\begin{enumerate}\item[(i)] $N^{(\beta)}$ is carried
by the set of paths $$\{\omega\in\mathbb D^+ \ | \ T_0(\omega) > 0 \ \text{and} \
X_t(\omega) = 0, \forall t \geq T_0\};$$
\item[(ii)] for every bounded measurable $f:[0,\infty)\to \re$ and each $t,s > 0$ and $\Lambda\in \mathcal G_t$ $$N^{(\beta)} (f(X_{t+s}), \Lambda\cap \{t < T_0\}) = N^{(\beta)}(\e_{X_t}(f(X_{s}), s<T_{0}), \Lambda \cap \{t < T_0\});$$

\item[(iii)] $N^{(\beta)}(1-e^{-T_0}) < \infty.$
\end{enumerate}
The entrance law associated to $N^{(\beta)},$ is defined by $$N^{(\beta)}_s(dy):=N^{(\beta)}(X_s\in dy, s<T_0), \quad
s>0.$$ According to the results in Lemma 2 in~\cite{rivero-cramer}, the entrance law $(N^{(\beta)}_{s}, s>0)$ is a $\beta$-ssel for $X.$ From Theorem~\ref{thm:1}, this is, up to a multiplicative constant, equal to that in (ii) in the op. cit. Theorem. In the case where Cram\'er's condition is satisfied, $\e(e^{\beta\xi_{1}}, 1<\zeta)=1,$ we have that $\lim_{t\to 0+}N^{(\beta)}_{t}1_{(a,\infty)}=0,$ for all $a>0,$ in which case we say that the recurrent extension leaves $0$ continuously. Whilst in the case $\e(e^{\beta\xi_{1}}, 1<\zeta)<1,$ we have  $\lim_{t\to 0+}N^{(\beta)}_{t}1_{(a,\infty)}>0,$ for all $a>0,$ we say that the recurrent extension leaves $0$ by a jump. In this case there is a {\it jumping in measure} $\eta$ such that $$\eta({d}x)= x^{-1-\gamma/\alpha}{d}x,
\qquad x>0,$$ and
$$\eta_{t}(dy)=\int_{(0,\infty)}\frac{dx}{x^{1+\gamma/\alpha}}\p_{x}(X_{t}\in dy, t<T_{0}),\qquad y>0,\ t>0.$$
Moreover, in the paper \cite{rivero-cramer} the process $X^{(\beta)}$ is constructed using It\^o's synthesis theorem. In what follows we will sketch the construction of another version of $X^{(\beta)}$, which will be such that its excursions are {\it colored}.  

We take a new process $\widetilde{X}$ taking values in $[0,\infty)\times\{-1,1\}$, for which $\{0\}\times\{-1,1\}$ is identified to a cemetery state, and it is such that when issued from $(x,y)\in (0,\infty)\times\{-1,1\},$ $\widetilde{X}_{t}=(X_{t},y),$ $t\geq 0,$ and $X$ is issued from $x.$ We understand $\widetilde{X}$ as a colored particle that moves, following the same dynamics as $X,$ and that take the color red, for 1, and blue, for $-1;$ and once the color is chosen it remains fixed until the absortion at zero of the particle. The semigroup of $\widetilde{X},$ say $(\widetilde{P}_{t}, t\geq 0),$ equals
$$\widetilde{P}_{t}((x,y), da\otimes dz)= \p_{x}(X_{t}\in da, t<T_{0})\otimes \delta_{y}(dz),\qquad (x,y), (a,b)\in [0,\infty)\times\{-1,1\},\ t\geq 0.$$  Let $\widetilde{N}^{(\beta)}$ be the measure on $\mathbb D^+\times\{-1,1\}$ defined by $\widetilde{N}^{(\beta)}=N^{(\beta)}\otimes (\frac{1}{2}\delta_{1}(dx)+\frac{1}{2}\delta_{-1}(dx)).$  Observe that $\widetilde{N}^{(\beta)}$ bears similar properties to those in (i)-(iii) above. 

Realize a marked Poisson point process $\Delta=((\Delta_s, U_{s}), s>0)$ on $\mathbb D^+\times\{-1,1\}$ with characteristic measure $\widetilde{N}^{(\beta)}.$ Thus each atom
$(\Delta_s, U_{s})$ is formed of a path $(\Delta_{s})$ and its mark $U_{s}.$ The marks are independent and follow a symmetric Bernoulli distribution. We will denote by $T_0(\Delta_s,U_{s})$ the lifetime of the path $(\Delta_{s}, U_{s})$,
i.e. $$T_0(\Delta_s, U_{s})=\inf \{t > 0 : \Delta_s(t)=0\}.$$ Set
$$\sigma_t=\sum_{s\leq t}T_0(\Delta_s, U_{s}),\qquad t\geq 0.$$ Since
$$\widetilde{N}^{(\beta)}(1-e^{-T_0})={N}^{(\beta)}(1-e^{-T_0})<\infty$$ it follows that for every $t>0,$ $\sigma_t<\infty$
a.s. It follows that the process $\sigma=(\sigma_t, t\geq 0)$ is
an increasing c\`adl\`ag process with stationary and independent
increments, i.e. a subordinator. Its law is characterized by its
Laplace exponent $\phi_{\beta},$ defined by
$$\e(e^{-\lambda\sigma_1})=e^{-\phi_{\beta}(\lambda)},\qquad \lambda>0,$$ and
$\phi_{\beta}(\lambda)$ can be expressed thanks to the L\'evy--Kintchine's
formula as $$\phi_{\beta}(\lambda)= \int (1-e^{-\lambda s})\nu_{\beta}(ds),$$ with
$\nu_{\beta}$ a measure such that $\int s \wedge 1 \ \nu_{\beta}(ds)<\infty,$
called the L\'evy measure of $\sigma;$ see e.g.
Bertoin~\cite{bertoin-subordinators} \S~3 for background. An application of
the exponential formula for Poisson point process gives
$$\er(e^{-\lambda\sigma_1})=e^{-\widetilde{N}^{(\beta)}(1-e^{-\lambda T_0})}=e^{-{N}^{(\beta)}(1-e^{-\lambda T_0})},\qquad \lambda>0,$$ i.e.
$\phi_{\beta}(\lambda)=\widetilde{N}^{(\beta)}(1-e^{-\lambda T_0})$ and the tail of the L\'evy
measure is given by $$\nu_{\beta}(s,\infty)=\widetilde{N}^{(\beta)}(s < T_0)=n_s 1,\qquad s>0.$$
Observe that if we assume $\phi_{\beta}(1)=\widetilde{N}^{(\beta)}(1-e^{-T_0})=1$ then $\phi_{\beta}$ is
uniquely determined. Since $\widetilde{N}^{(\beta)}$ has infinite mass, $\sigma_t$ is
strictly increasing in $t.$ Let $L_t$ be the local time at $0,$
i.e. the inverse of $\sigma$ $$L_t=\inf\{ r > 0 :
\sigma_r>t\}=\inf\{r>0 : \sigma_r\geq t\}.$$ Define a process
$(\widetilde{X}^{(\beta)}_t, t \geq 0)$ as follows. For $t\geq 0,$ let $L_t=s,$
then  $\sigma_{s-}\leq t\leq \sigma_{s},$ set
\begin{equation}\label{itoprogram}\widetilde X^{(\beta)}_t=\begin{cases}
(\Delta_s(t-\sigma_{s-}), U_{s}) &\quad\text{if}\quad \sigma_{s-}<\sigma_{s}\\
(0,0) &\quad\text{if}\quad \sigma_{s-}=\sigma_{s}\ \text{or}\ s=0.
\end{cases}\end{equation}
That the process so constructed is a Markov process taking values in $[0,\infty)\times\{-1,0,1\}$ is a consequence of the main results in~\cite{MR705615} and \cite{MR859838}. The arguments needed to check that the hypotheses in these papers are satisfied are an elementary extension of those given in the paper~\cite{rivero-cramer} to verify that the corresponding conditions are satisfied by $X.$ We will denote by $\widetilde{\p}^{(\beta)}$ its law.

It is easily verified, from the construction, that the projection of $\widetilde{X}^{(\beta)}$ in its first coordinate, is a version of the self-similar Markov process $X^{(\beta)}$.  For $\beta^{\prime}<\beta,$ we will next construct a version of $X^{(\beta^{\prime})}$ from $\widetilde{X}^{(\beta)}.$ The rest of this section will use facts from the fluctuation theory of L\'evy processes, we refer to \cite{bertoinbook} for background on this topic.

Fix $q > 0,$ and let $A^{+},$ $A^{-,q}$ be the additive functionals of $\widetilde{X}^{(\beta)}$
defined by
$$A^{+}_t=\int^{t}_0 1_{\{\widetilde{X}^{(\beta)}_{s}\in (0,\infty)\times\{1\}\}} {d}s, \qquad A^{-,q}_t=\int^{t}_0 q^{\alpha/\beta} 1_{\{\widetilde{X}^{(\beta)}_s\in(0,\infty)\times\{-1\}\}} {d}s,\qquad t\geq
0,$$ and we introduce the time change $\tau^{(q)},$ which is the
generalized inverse of the fluctuating additive functional
$A^{+}-A^{-,q},$ that is
$$\tau^{(q)}(t)=\inf\{s>0: A^{+}_s-A^{-,q}_s > t\},\qquad t\geq 0,\qquad
\inf{\emptyset}=\infty.$$ Now, let $Y^{(\beta)}$ to be the projection of the process $\widetilde{X}^{(\beta)}$ in the first coordinate and $Y^{(+,q)}$ be the process $Y$  time changed by $\tau^{(q)},$
\begin{align*}
Y^{(+,q)}_t=\begin{cases} Y_{\tau^{(q)}(t)}
& \text{if}\ \tau^{(q)}(t)<\infty,\\
\Delta & \text{if}\ \tau^{(q)}(t)=\infty,\\
\end{cases}
\end{align*}
where $\Delta$ is a cemetery or absorbing state. Notice that the time change has the effect of deleting all the blue paths in $\widetilde{X}^{(\beta)},$ together with some red paths. So, the process, $Y^{(+,q)}$, is obtained by pasting together red paths, to which we have deleted a random length in its starting part, and the distribution of the deleted length depends on $q.$  A more precise description is given in the following Lemma. 
\begin{lemma}\label{lemma:1const}Under $\widetilde{\p}^{(\beta)}$ the process $Y^{(+,q)},$ is a
positive $\alpha$-self-similar Markov process for which $0$ is a
regular and recurrent state and that leaves $0$ by a jump according
to the jumping-in measure
$c_{\alpha,\beta,\rho}\eta_{\rho\beta},$ with
$$\eta_{\beta\rho}({d}x)= x^{-1-\rho\beta}{d}x,
\qquad x>0,$$ where $\rho$ is given by
$$\rho=\frac{1}{2}+\frac{\alpha}{\pi\beta}\arctan\left(\frac{1-q}{1+q}\tan\left(\frac{\pi\beta}{2\alpha}\right)\right)\in]0,1[,$$
and $0<c_{\alpha,\beta,\rho}=\frac{\beta(1-\rho)}{2}N^{(\beta)}(X^{\beta\rho}_1,
1<T_0)<\infty,$ is a constant.
\end{lemma}
Before proving this lemma let us observe that the above construction has as a consequence that we can actually embed, via a time change, all the recurrent extensions into a single one, namely the one corresponding to the ssel with largest self-similarity index. In the case where Cram\'er's condition is satisfied, all the recurrent extensions that leave zero by a jump can be embedded into the one that leaves zero-continuously. It will be seen more clearly in the proof of the latter Lemma that this embedding arises by deleting the beginning part of the excursions, and the deleted proportion length of paths increases as the self-similarity index of the entrance law decreases. Another interpretation of the identity (\ref{entrancelaws}) is given in terms of the meander process. The meander process of length $r$ is defined as the path of the excursion process in $(0,r],$ conditioned to live for a period of time of length at least $r$. So the law of the meander at time $1$ of $X^{(\beta)}$ is $$N^{(\beta)}(X_{1}\in dy | 1<T_{0}).$$ Similarly, that of $Y^{+,q}\stackrel{\text{Law}}{=}X^{(\beta\rho)},$ is given by 
$$N^{(\beta\rho)}(X_{1}\in dy | 1<T_{0}),$$  with $\rho$ as in the previous Lemma. According to the identity in (\ref{entrancelaws}) the above are related by the formula
\begin{equation}
\begin{split}
N^{(\beta\rho)}(X_{1}\in dy | 1<T_{0})&=N^{(\beta)}\left(\frac{X_{1}}{\mathcal{B}^{1/\alpha}_{\frac{\beta\rho}{\alpha}, \frac{\beta(1-\rho)}{\alpha}}}\in dy | 1<T_{0}\right)\\
&=N^{(\beta)}\left(X_{\frac{1}{\mathcal{B}_{\frac{\beta\rho}{\alpha}, \frac{\beta(1-\rho)}{\alpha}}}}\in dy \Big| \frac{1}{\mathcal{B}_{\frac{\beta\rho}{\alpha}, \frac{\beta(1-\rho)}{\alpha}}}<T_{0}\right),
\end{split}
\end{equation}
where the final identity is a consequence of the self-similarity property of the excursion measure $N^{(\beta)},$ which is in turn inherited from that of $X$ and the one of the ssel. Said otherwise, the position at time $1$ of a generic excursion from zero for the process $X^{(\beta\rho)}$ corresponds to a position at an independent random time $1/{\mathcal{B}_{\frac{\beta\rho}{\alpha}, \frac{\beta(1-\rho)}{\alpha}}}$ of a generic excursion from zero of the process $X^{(\beta)}.$ This observed delay is a consequence of the time change.   

To prove the Lemma~\ref{lemma:1const} we observe that  the process $Y^{(+,q)}$ is a pssMp, as it can be easily verified using standard arguments. We should prove that the measure $N^{\beta,+,q}$ of the
excursions from $0$ of $Y^{(+,q)},$ is such that $N^{\beta,+,q}(Y_0\in
{d}y)=c_{\alpha,\rho\beta}\eta_{\rho\beta}({d}y).$ This will be a consequence of the following auxiliary Lemma.
\begin{lemma}\label{stable}\begin{enumerate}
\item[(i)]The processes $Z^+, Z^{-,q}$ defined by
$$Z^{+}\equiv (Z^{+}_t=A^{+}_{L^{-1}_t}, t \geq 0);\qquad Z^{-,q}\equiv (Z^{-,q}_t=A^{-,q}_{L^{-1}_t}, t \geq 0) $$ are independent stable
subordinators of parameter $\beta/\alpha,$ and their respective
L\'evy measures are given by
$\pi^+({d}x)=\frac{c}{2}x^{-1-\beta/\alpha}{d}x,$ and
$\pi^{-,q}({d}x)=\frac{qc}{2}x^{-1-\beta/\alpha}{d}x,$ on
$]0,\infty[,$ and $c\in]0,\infty[$ is a constant.
\item[(ii)]The process $Z=Z^+-Z^{-,q}$ is a stable process with
parameter $\beta/\alpha$, positivity parameter
$\rho=\pr(Z_1>0),$
 with $\rho$ as defined in Lemma~\ref{lemma:1const},
 and L\'evy measure $$\Pi_{Z}(dx)=\pi^{+}(dx)1_{\{x>0\}}+\pi^{-,q}(-dx)1_{\{x<0\}}.$$ 
 \item[(iii)] The upward and downward ladder height processes, $H$
and $\widehat{H},$ associated to $Z$ are stable subordinators of
parameter $\beta\rho/\alpha$ and $\frac{\beta(1-\rho)}{\alpha},$
respectively.
\end{enumerate}
\end{lemma}
\begin{proof}
Because the stable subordinator $L^{-1}$ has the L\'evy measure $N^{(\beta)}(T_0\in{d}t),$ it follows that
$N^{(\beta)}(T_0\in{d}t)=-d\eta^{(\beta)}_{t}1=ct^{-1-\beta/\alpha}{d}t,$ $t>0,$
for some constant $0<c<\infty.$ That the processes $Z^+,Z^{-,q},$ are independent subordinators is a
standard result in the theory of excursions of Markov processes and
follows from the fact that they are defined in terms of the atoms of the Poisson point process $((\Delta_{s}, U_{s}), s>0)$ whose mark equals $1$ and $-1$, respectively. The
self-similarity property for $(Z^+,Z^{-,q})$ follows from that of
$(\widetilde{X}^{(\beta)}, \widetilde{\p}^{(\beta)}).$ The jump measure $\pi^+$ of $Z^+$ is given
by
$$\pi^+({d}t)=\frac{1}{2}N^{(\beta)}(T_0\in {d}t),$$ while that of
$Z^{-,q}$ is $$\pi^{-,q}({d}t)=\frac{q}{2}N^{(\beta)}(T_0\in {d}t).$$ This is a consequence of  the following calculations based on the exponential formula for Poisson point processes:
\begin{equation*}
\begin{split}
\widetilde{\e}^{(\beta)}\left(\exp\{-\lambda Z^{-,q}_{1}\}\right)&=\exp\left\{-\int_{(0,\infty)}(1-e^{-\lambda q^{\alpha/\beta}t}) \widetilde{N}^{(\beta)}(T_{0}\in dt, U=-1)\right\}\\
&=\exp\left\{-\int_{(0,\infty)}(1-e^{-\lambda q^{\alpha/\beta}t}) \frac{1}{2}{N}^{(\beta)}(T_{0}\in dt)\right\}\\
&=\exp\left\{-\int_{(0,\infty)}(1-e^{-\lambda q^{\alpha/\beta}t}) \frac{c}{2}\frac{dt}{t^{1+\frac{\beta}{\alpha}}}\right\}\\
&=\exp\left\{-\int_{(0,\infty)}(1-e^{-\lambda s}) \frac{q}{2}N^{(\beta)}(T_{0}\in ds)\right\},
\end{split}
\end{equation*}
for all $\lambda \geq 0.$
The proof of the assertion in (ii)
is straightforward. It is well known in the fluctuation theory of
L\'evy processes that the upward and downward ladder height
subordinators associated to a stable L\'evy process have the form
claimed in (iii). See for instance \cite{bertoinbook} Chapter VIII.
\end{proof}

\begin{proof}[Proof of Lemma~\ref{lemma:1const}]
By construction, the closure of the set of times at which the process
$Y^{(+,q)}$ visits $0$ is the regenerative set given by the closure
of the image of the supremum of the stable L\'evy process $Z.$ This coincides with the image of the upward ladder height subordinator $H,$ which is a $\beta\rho/\alpha$-stable subordinator. This set is an unbounded perfect regenerative set with zero Lebesgue measure, see \cite{bertoin-subordinators} Chapter 2. It follows that $0$ is a regular and recurrent state for $Y^{(+,q)}.$ The length of any excursion out of $0$ for $Y^{(+,q)}$ is distributed as a jump of $Z$ to reach a new supremum or, equivalently, as a jump of the upward ladder height process $H$ associated to $Z.$ Let $\overline{N}$ denotes the
measure of the excursions from $0$ of $\overline{Z}-Z,$ the process
$Z$ reflected at its current supremum, viz.
$$(\overline{Z}-Z)_t:=\sup_{0\leq s\leq t}\left\{0\vee Z_s\right\} - Z_t,\qquad t\geq 0;$$ and let $R$ denote the lifetime of the generic
excursion from $0$ of $\overline{Z}-Z$. Let $\Pi_Z$ be the L\'evy
measure of $Z$ and $\widehat{V}$ be the renewal measure of the
downward ladder height subordinator $\widehat{H},$ that is
$$\widehat{V}({d}y)=\er\left(\int^{\infty}_{0}1_{\{\widehat{H}_s\in{d}y\}}{d}s\right).$$ It is known in the fluctuation theory for L\'evy processes that
under $\overline{N}$ the joint law of $Z_{R-}$ and $Z_{R}-Z_{R-}$ is given by
$$\overline{N}(Z_{R-}\in {d}x, -(Z_{R}-Z_{R-})\in{d}y)=
\widehat{V}({d}x)\Pi_{Z}({d}y)1_{\{0<x<y\}},$$ see for instance \cite{kyprianoubook}, Chapter VII. Moreover, the L\'evy measure of $H,$ say $\Pi_{H}({d}x),$ is such
that
$$\Pi_{H}]x,\infty[=\int\int_{\{0\leq s\leq u\}}\widehat{V}({d}s)\Pi_Z({d}u)1_{]x,\infty[}(u-s),\qquad x>0.$$
cf. \cite{thesevigon}. In our framework, $-(Z_{R}-Z_{R-})$ denotes
the length of the generic positive excursion from $0$ for
$(Y,\widetilde{\p}^{(\beta)})$ and $Z_{R-}$ is the length of the portion
of the generic positive excursion from $0$ of $(Y,\widetilde{\p}^{(\beta)})$
that is not observed while observing a generic excursion from $0$ of
$Y^{(+,q)}.$ See Figure~1. Furthermore, $-(Z_{R}-Z_{R-})-Z_{R-}$ is the length of
the generic excursion from $0$ for $Y^{(+,q)},$ so
$N^{\beta,+,q}(T_0\in{d}t)=\Pi_{H}({d}t).$

\begin{figure}
\centering
\includegraphics[width=8cm]{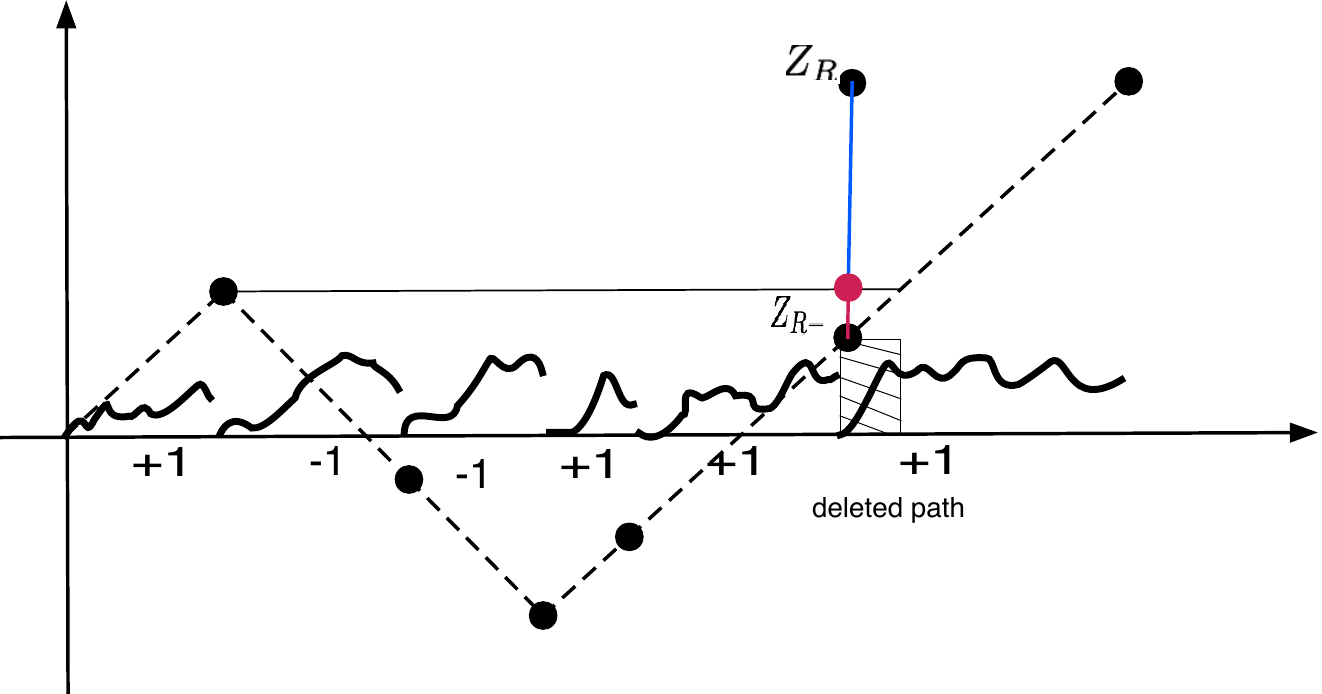}\label{fig1}
\caption{A schematisation of the path of the processes $\widetilde{X}^{(\beta)},$ and $A^{+}-A^{-,q}$. The dots show the jumps of the process $Z.$ The shaded area represents the portion of path that is deleted with the time change. }
\end{figure}

Let $\widetilde{N}^{(\beta)}( \cdot|T_0=\cdot)$ denote a version of the regular
conditional law of the generic excursion under $\widetilde{N}^{(\beta)}$ given the
lifetime $T_0.$ Similarly, the notation
$N^{\beta,+,q}(\cdot|T_0=\cdot)$ will be used for the analogous
conditional law under $N^{\beta,+,q}.$ These laws can be constructed
using the method in~\cite{chaumont97}, see also \cite{rivero-cramer}.

Finally, it follows from the verbal description above, that for any
positive and measurable function $f:\re\to\re^+$
\begin{equation*}
\begin{split}
N^{\beta,+,q}(f(Y_0))&=\int^{\infty}_{0}N^{\beta,+,q}(T_0\in{d}u)N^{\beta,+,q}\left(f(Y_0)| T_0=u\right)\\
&=\int_{t\in]0,\infty[}\int_{s\in]t,\infty[}\overline{N}(Z_{R-}\in {d}t, -(Z_{R}-Z_{R-})\in{d}s)\widetilde{N}^{(\beta)}(f(Y_{t})|T_0 = s, U=1)\\
&=\frac{1}{2}\int_{t\in]0,\infty[}\widehat{V}({d}t)\int_{s\in]t,\infty[}\Pi_Z({d}s)\widetilde{N}^{(\beta)}(f(Y_{t}) |T_0 = s, U=1)\\
&=\int_{t\in]0,\infty[}\widehat{V}({d}t)\int_{s\in]t,\infty[}\widetilde{N}^{(\beta)}(T_0\in{d}s, U=1)\widetilde{N}^{(\beta)}(f(Y_{t}) |T_0 = s, U=1)\\
&=\int_{t\in]0,\infty[}\widehat{V}({d}t)\widetilde{N}^{(\beta)}(f(Y_{t}), t < T_0, U=1).\\
\end{split}
\end{equation*}
Given that the downward ladder height subordinator $\widehat{H}$ is
a stable process with index $\frac{\beta(1-\rho)}{\alpha},$ it follows
that
\begin{equation*}
\begin{split}
N^{\beta,+,q}(f(Y_0))&=\frac{\beta(1-\rho)}{\alpha}\widehat{k}\int_{t\in]0,\infty[}{d}t  t^{\frac{\beta(1-\rho)}{\alpha}-1}\widetilde{N}^{(\beta)}(f(Y_{t}), t < T_0, U=1)\\
&=\frac{\beta(1-\rho)}{\alpha}\frac{\widehat{k}}{2}\int_{t\in]0,\infty[}{d}t  t^{\frac{\beta(1-\rho)}{\alpha}-1}t^{-\beta/\alpha}N^{(\beta)}(f(t^{1/\alpha} Y_{1}), 1 < T_0 )\\
&=\frac{\beta(1-\rho)}{\alpha}\frac{\widehat{k}}{2}\int_{]0,\infty[}N^{(\beta)}(Y_1\in{d}x, 1<T_0)\int_{t\in]0,\infty[}{d}t  t^{-\frac{\beta\rho}{\alpha}-1}f(t^{1/\alpha} x)\\
&=\beta(1-\rho)\frac{\widehat{k}}{2}N^{(\beta)}\left(Y^{\beta\rho}_1, 1<T_0\right)\int_{u\in]0,\infty[}{d}u  u^{-\beta\rho-1}f(u)\\
&=c_{\alpha,\rho\vartheta}\eta_{\rho}f,
\end{split}
\end{equation*}
where $\widehat{k}$ is a constant that depends on the normalization
of the local time at zero for the reflected process
$\overline{Z}-Z;$ which without loss of generality can, and is
supposed to be  $\widehat{k}=1.$ The finiteness of $N^{(\beta)}\left(Y^{\beta\rho}_1, 1<T_0\right)$ follows from Theorem~\ref{thm:1} because we have $$N^{(\beta)}\left(Y^{\beta\rho}_1, 1<T_0\right)=\widehat{\e}^{(\beta)}\left(I^{-\frac{\beta\rho}{\alpha}}I^{\frac{\beta}{\alpha}-1}\right)=\widehat{\e}^{(\beta)}\left(I^{\frac{\beta(1-\rho)}{\alpha}-1}\right),$$ which happens to be finite because $$\widehat{\e}^{(\beta)}\left(\exp\left\{\frac{\beta(1-\rho)}{\alpha}(\alpha\xi_{1})\right\},1<\zeta\right)={\e}\left(\exp\left\{\beta\rho\xi_{1}\right\},1<\zeta\right)\leq 1.$$  \end{proof}
\begin{remark}The previous proof is inspired in~\cite{rogers-williams}.\end{remark}

\section{Two extensions}\label{extensions}
In this section we state without proof two results that can be obtained with essentially the same proof as that of Theorems~\ref{thm:1} and \ref{thm:20}.

\subsection{A result in the fluctuation theory of self-similar Markov processes }\label{fluctuation}
This subsection is motivated by the work \cite{ckpr}. Let $(Z,h)$ be a bivariate L\'evy process such that its coordinates are subordinators. Denote by $\Pi_{Z},$ $\Pi_{h}$ the L\'evy measures of $Z$ and $h,$ respectively. We will use the following notation for tail L\'evy measures $$\overline{\Pi}_{Z}(x)=\Pi_{Z}(x,\infty),\quad \overline{\Pi}_{h}(x)=\Pi_{h}(x,\infty),\quad A_{h}(x)=\max\{1,\overline{\Pi}_{h}(1)\}+\int^x_{1}\overline{\Pi}_{h}(z){d}z,\qquad x>0.$$   For $\alpha>{0}$, let $\tau_{h}$ be the time change defined as $$\tau_{h}(t)=\inf\{s>0 : \int^s_{0}e^{\alpha h_{s}}ds>t\},\qquad t>0.$$ Define a stochastic process $(V_{t}, t\geq 0)$ pathwise as the Stieljets integral of $e^{h_{s}}$ with respect to $Z$, $$V_{t}=\int^t_{0}e^{\alpha h_{s-}}{d}Z_{s},\qquad t\geq 0.$$ For $a\geq 0,$ $x>0$  we denote by $\mathbf{Q}_{a,x}$ the law of the processs  $(R,H)$  defined as follows 
$$R_{t}=a+x^{\alpha}\int^{\tau_{h}(tx^{-\alpha})}_{0}e^{\alpha h_{s-}}{d}Z_{s},\quad  H_{t}=xe^{h_{\tau_{h}(tx^{-\alpha})}},\qquad  t\geq 0.$$ 

\begin{teo}\label{thm:2}
Under $(\mathbf{Q}_{a,x}, a\geq0, x>0)$ $(R,H)$ is a Feller process in $[0,\infty)\times(0,\infty).$  Assume that \begin{equation}\label{LMcond}
\int^\infty_{e}\frac{\log(y)}{A_{h}(\log(y))}\Pi_{Z}({d}y)<\infty.\end{equation}
 The random variable $\widetilde{I}:=\int^\infty_{0}e^{-\alpha h_{s}}{d} Z_{s-}$ is finite a.s. The family of measures $(\widetilde{\mu}_{t}, t\geq 0)$ defined by $$\widetilde{\mu}_{t}f=\er\left(f\left(\frac{t \widetilde{I}}{I_{h}}, \left(\frac{t}{I_{h}}\right)^{1/\alpha}\right)\frac{1}{I_{h}}\right),$$ form an entrance law for $\mathbf{Q}_{\cdot,\cdot},$ where $I_{h}:=\int^\infty_{0}e^{-\alpha h_{s}}{d}s.$
\end{teo}
The condition (\ref{LMcond}) has been introduced by Linder and Maller~\cite{lindner+maller}, and it is a necessary and sufficient  for the L\'evy integral, $\int^\infty_{0}e^{-\alpha h_{s-}}{d} Z_{s},$ to be finite. Their results hold not only for subordinators but for any L\'evy process. It is important to mention that if $\er(h_{1})<\infty$ then the condition (\ref{LMcond}) is equivalent to $\er(\log^+(Z_{1}))<\infty$ which is a well known equivalent condition for the convergence a.s. of the integral $\int^\infty_{0}e^{-s}{d}Z_{s},$ see e.g. \cite{Satobook}. It is important to mention that it is possible to establish a result similar to Theorem~\ref{thm:2} for more general couples of L\'evy processes but, as we do not have any application in mind, we wont pursue this line of research. 

In the work \cite{ckpr} the process $h$ is the upward ladder height associated to a L\'evy process $\xi$ and $Z$ is another subordinator constructed as functionals of the excursions of $\xi$ reflected in its past supremum. This processes are key elements to develop  a fluctuation theory for pssMp. We refer to \cite{ckpr} for further details.

That the process $(R,H)$ is a Feller process in $[0,\infty)\times(0,\infty),$ is obtained using arguments similar to those in \cite{ckpr}. The rest of the proof follows along the same lines of Theorem~\ref{thm:1}; mainly the roll played by the L\'evy process $\xi$ is played by the bivariate L\'evy process $(Z,h)$. It is known that Lebesgue measure $\Lambda_{2}$ in $\re^2$ is an invariant measure for any bivariate L\'evy process, so for $(Z,h).$ We take $(Y=(Y_{1},Y_{2}),\mathbb{Q})$ the Kuznetzov process associated to $(Z,h)$ and $\Lambda_{2}.$  This process has the Markov property under $\mathbb{Q}$ and is invariant under translations. The birth and death time of $Y$ under $\mathbb{Q}$ are infinite because the Lebesgue measure is invariant for $(Z,h)$ and $(\widehat{Z},\widehat{h})=-(Z,h).$  The rest of the proof is essentially the same as that of Theorem~\ref{thm:1}

\subsection{Multi-self-similar Markov processes}\label{multi}
The following definition steams from the work of Jacobsen and Yor~\cite{JACY}. 
\begin{defi}[Jacobsen and Yor~\cite{JACY}]
A $n$-dimensional Markov process $X$ with state space $\re^{n}_{+}=[0,\infty)^{n}$ is $1/{\alphab}$-multi-self-similar, with $\alphab=(\alpha_{1},\ldots,\alpha_{n})\in \re^{n},$ if for all scaling factors $c_{1},\ldots, c_{n}>0,$ and all initial states $x=(x_{1},\ldots, x_{n})\in\re^{n}_{+}$ it holds that 
\begin{equation}\label{MSSMP}
\left(\left(\left\{c_{i}X^{i}_{t/c}\right\}_{i\in(1,\ldots,n)}, t\geq 0\right), \p_{(x_{1},\ldots, x_{n})}\right)\stackrel{\text{Law}}{=}\left(\left(\left\{X^{i}_{t}\right\}_{i\in(1,\ldots,n)}, t\geq 0\right), \p_{(c_{1}x_{1},\ldots, c_{n}x_{n})}\right),
\end{equation}
where $c=\prod^{n}_{i=1}c^{\alpha_{i}}_{i}.$ We denote by $T_{\mathbf 0}=\inf\{t>0: X^{i}_{t}=0,\ \text{for some } i\in\{1,\ldots,n\} \}.$
\end{defi}
Note that for $n>1,$ the symbol $1/{\alphab}$ is senseless, but we made the choice of using it to preserve the customary notation for the $1$-dimensional case. Of course, in the case $n=1,$ $1/0$ neither makes sense but observe that in the definition (\ref{MSSMP}) this is does not cause any inconvenient. 
 
Examples of multi-self-similar diffusions where introduced by Jacobsen~\cite{jacobsen} and Jacobsen and Yor \cite{JACY}. Examples of processes with jumps can be easily obtained from these processes by subordination via independent stable subordinators.

Extending the arguments of Lamperti~\cite{Lamperti}, Jacobsen and Yor~\cite{JACY} established that there exists a one to one correspondence between multi--self--similar Markov
processes on $\re^{n}_{+}$ and L\'evy processes taking values in $\re^{n}\cup\{\Delta\},$ that we next sketch. For that end we start by setting some notation. 

For ${\alphab}\in \re^{n}$ fixed, we will denote $\mathit{p}_{\alphab}(u)=\prod^{n}_{i=1}u^{\alpha_{i}}_{i},$ for all $u=(u_{1},\ldots, u_{n})\in (0,\infty)^{n}.$ For $u,v\in [0,\infty)^{n}$ we denote by $u\circ v$ the Hadamard product of $u$ and $v,$ $u\circ v=(u_{1}v_{1},\cdots, u_{n}v_{n}).$ For a vector $z=(z_{1},\ldots,z_{n})\in \re^{n}\cup\{\Delta\}$ we will denote by $\mathscr{E}(z)=\left(\exp\{z_{1}\},\ldots, \exp\{z_{n}\}\right),$ if $z\in\re^{n}$ and $\mathscr{E}(\Delta)\in{\mathbf 0}:=\{u\in[0,\infty)^{n} :\prod^{n}_{i=1}u_{i}=0\}.$ Assume now that $(\mathbb D,\mathcal D)$ is the space of c\`adl\`ag paths
$\omega: [0,\infty[ \to \re^{n}\cup\{\Delta\},$ endowed with the
$\sigma$--algebra generated by the coordinate maps and the completed natural
filtration $(\mathcal D_t, t\geq 0).$ Let $\p$ be a probability
measure on $\mathcal D$ such that under $\p$ the
process $\xi$ is a L\'evy process that takes values in $\re^{n}\cup\{\Delta\}.$ Where the state $\Delta$ is understood as a cemetery state, and so the first hitting time of $\Delta$ or life time for $\xi,$ say $\zeta,$ follows an exponential distribution with some parameter ${q}\geq 0,$ and the value $q=0,$ is permitted to include the case where $\zeta=\infty,$ $\p$-a.s.
 
For ${\alphab}\in \re^{n},$ set for $t\geq 0$
$$\tau(t)=\inf\{s>0, \int^s_0 e^{<\alphab,\xi_r>}dr > t \},$$
with the usual convention that $\inf\{\emptyset\}=\infty.$  For $x\in [0,\infty)^{n},$ let $(X^{x}_{t}, t\geq 0)$ be the process defined by {\it Lamperti, Jacobsen and Yor  transformation}:
\begin{equation}\label{eq:1}
X^{(x)}_{t}:=\begin{cases}x\circ\mathscr{E}\left(\xi_{\tau(t/\mathit{p}_{\alphab}(x))}\right), & \text{if} \ t<\mathit{p}_{\alphab}(x)\int^{\zeta}_{0}\exp\{<\alphab, \xi_{s}>\}ds\\
\overline{0} &  \text{if} \ t\geq \mathit{p}_{\alphab}(x)\int^{\zeta}_{0}\exp\{<\alphab, \xi_{s}>\}ds\ \text{or} \ x\in{\mathbf 0}.
\end{cases}
\end{equation} for $t\geq 0.$ For $x\in [0,\infty)^{n},$ we denote by $\p_{x}$ the law of $X^{(x)},$ that is the image measure of $\p$ under Lamperti's transformation applied to the process $\xi$.  It is a standard fact that this process is adapted with respect to the filtration $\mathcal{F}_{t}=\mathcal{D}_{\tau(t)},$ $t\geq 0,$ and inherits the strong Markov property from $\xi,$ with respect to the filtration $(\mathcal{F}_{t}, t\geq 0),$ see for instance {\cite{Rogersandwilliams}}. A straight forward verification shows that the Markov family $(X,\p_{x})_{x\in[0,\infty)^{n}}$ bears the $1/\alphab$-multi-self-similar property. Jacobsen and Yor proved that any multi-self-similar Markov process that never hits the set ${\mathbf 0},$ can be constructed this way. A perusal of their proof and that of Lamperti for the case of dimension $1,$ shows that the result can be easily extended to any multi-self-similar Markov process killed at its first hitting time of ${\mathbf 0}.$ We do not include the details. We have the following result that extends the Theorem~\ref{thm:20} for the class of multi-self-similar Markov processes.

\begin{teo}
Let $\alphab=(\alpha_{1},\ldots, \alpha_{n})\in\re^{n}$, $X=\left((X_{t})_{t\geq 0},(\p_{x})_{x\in \re^{+}_{n}}\right)$ be a $1/\alphab$-multi-self-similar Markov process  and $\xi$ the $\re^{n}$-valued L\'evy process associated to it via the Lamperti-Jacobsen-Yor transformation.  Assume that $\xi$ has an infinite lifetime and $\lim_{t\to \infty}<\alpha,\xi_{t}>=\infty.$ There exists a unique entrance law $(\mu_{t}, t\geq 0)$ for $X$ whose $\lambda$-potential is given by 
\begin{equation*}
\begin{split}
&\int^{\infty}_{0}dt e^{-\lambda t}\mu_{t}f\\
&=\int_{(0,\infty)^{n}}m(dx_{1}\cdots dx_{n})f(x_{1},\cdots,x_{n}){\e}\left(\exp\{-\lambda p_{\alphab}\left((x_{1},\ldots, x_{n})\right)\int^{\infty}_{0}\exp\{-<\alphab,\xi_{s}>\}ds\}\right),
\end{split}
\end{equation*}where $f:(0,\infty)^{n}\to[0,\infty)$ is a measurable function and $$m(dx_{1}\cdots dx_{n})=dx_{1}\cdots dx_{n}\prod^{n}_{i=1} (x_{i})^{\alpha_{i}-1},\qquad \text{on} \ (0,\infty)^{n}.$$
This entrance law has the scaling property: for $c_{1},\ldots, c_{n}>0,$ $c=\prod^{n}_{i=1}c^{\alpha_{i}}_{i}$ 
\begin{equation}\label{scalingmulti}
\mu_{t/c}\mathbf{H}_{(c_{1},\ldots,c_{n})}f=\mu_{t}f,\qquad t>0,\end{equation} 
with $$\mathbf{H}_{(c_{1},\ldots, c_{n})}f(x_{1},\ldots,x_{n})=f(c_{1}x_{1},\ldots, c_{n}x_{n}),\qquad (x_{1},\ldots, x_{n})\in (0,\infty)^{n}$$
\end{teo}
The proof of this result follows essentially the same argument as that of Theorem~\ref{thm:1}. For, it is necessary first to recall that Lebesgue's measure in $\re^{n},$ say $\Lambda,$ is an invariant measure of $\xi$, and second that $\xi$ and $\widehat{\xi}=-\xi,$ are in weak duality with respect to it. Then we construct the Kusnetzov process associated to $\Lambda.$ This measure is invariant under translations. The rest of the proof is similar. The scaling property can be verified as in the proof of Theorem~\ref{thm:1} using the invariance under  translations. 


\end{document}